\documentclass[11pt]{article}
\usepackage{dsfont}
\usepackage{amssymb}
\usepackage{geometry}
\usepackage{amsmath}
\usepackage{graphicx}
\usepackage{subcaption}
\usepackage{listings}
\usepackage{amsthm}
\usepackage{hyperref}
\usepackage{hhline}
\usepackage{wrapfig}
\usepackage{setspace}  
\usepackage{pstricks-add}
\usepackage[rflt]{floatflt}
\usepackage{enumerate}
\usepackage{verbatim}
\usepackage{mathtools}
\usepackage{float}
\usepackage{color}
\usepackage{enumitem} 
\graphicspath{{converted_graphics/}}
\newtheorem{teo}{Theorem}
\newtheorem{cor}[teo]{Corollary}

\newtheorem{prop}[teo]{Proposition}

\theoremstyle{definition}
\newtheorem{defi}[teo]{Definition}

\newtheorem{ejem}[teo]{Example}

\newcommand{\C}{\mathcal{C}}
\newcommand{\A}{\mathcal{A}}
\newcommand{\B}{\mathcal{B}}

\newcommand{\Tr}{\operatorname{Tr}}
\newcommand{\tr}{\operatorname{tr}}
\newcommand{\EV}{\operatorname{EV}}
\newcommand{\Du}{\operatorname{D}(\omega)}

\usepackage{listings}
\hypersetup{
    colorlinks=true,
    linkcolor=blue,
    filecolor=magenta,      
    urlcolor=cyan,
}

\begin{document}
	\onehalfspace

	\title{\Large{Polynomial with cyclic monotone elements with applications to random matrices with discrete spectrum}}
\author{\large{Octavio Arizmendi, Adri\'an Celestino}}
\date{\today}
	{\maketitle}
	\begin{abstract}
	We provide a generalization and new proofs of the formulas of Collins, Hasebe and Sakuma for the spectrum of polynomials in cyclic monotone elements. This is applied to random matrices with discrete spectrum.
	\end{abstract}

\section{Introduction}

Free Probability Theory  has been used widely to describe the behavior of Random Matrices of large size. The first result comes from Voiculescu's paper \cite{Voi91}, where he discovered asymptotic freeness of Gaussian matrices. Other variations and extensions of free probability such as Second Order Freeness \cite{MSS}, Traffic Freeness \cite{Male},  Infinitesimal Freeness \cite{Shl} and Matricial Freeness \cite{Len} have found applications to Random Matrix Theory.

A new notion of independence, \emph{cyclically monotone independence}, was introduced recently by Collins, Hasebe and Sakuma in \cite{CHS}, where they show that it describes the limiting joint distribution of a pair of families of random matrices $A=\{A_1,....,A_n\}$, $B=\{B_1,....,B_n\}$.
The main difference between previous results is that the set of random matrices $A$ is assumed to converge w.r.t the trace to a set of trace class operators.  This  paper  is  a  continuation  of \cite{CHS}. Our main results extend the applicability of the theory by describing how to obtain the limiting eigenvalues for a large class of polynomials in cyclic monotone variables.

To be more precise, apart from an asymptotic result on Random Matrices, the authors of \cite{CHS} found, quite explicitly, the set eigenvalues for some polynomials of degree $2$ and $3$ in cyclically monotone elements. As one example,  let us state their result about the eigenvalue set, denoted by EV, for the commutator and anticommutator.
\begin{teo}\cite{CHS}
\label{MainCHS}
Let $(\A,\omega,\tau)$ be a non-commutative probability space with tracial weight $\omega$. Consider $a\in \Du$ and $b \in\A$  such that $a$ is a trace class operator with respect to $\omega$ and suppose that $a,b$ are selfadjoint and the pair $(a,b)$ is cyclically monotone with respect to $(\omega,\tau)$.
\begin{enumerate}

\item If $p=\sqrt{\tau(b^2)}+\tau(b)$ and $q = -\sqrt{\tau(b^2)}+ \tau(b)$  then
$$\operatorname{EV}(ab+ba) = (p\operatorname{EV}(a))\sqcup (q\operatorname{EV}(a)).$$
\item If $r = \sqrt{\tau(b^2)-\tau(b)^2}$, then
$$\operatorname{EV}(\operatorname{i}(ab-ba)) = (r\operatorname{EV}(a))\sqcup (-r\operatorname{EV}(a)).$$
\end{enumerate}
\end{teo}

The original proof of the above theorem was done by combinatorial means, calculating direclty the moments of the above polynomial. As mentioned above, in this paper we extend widely the set of polynomials $P(a_1, a_n, b_1,\dots,b_k)$ for which we can calculate the eigenvalue set from the individual joint distributions of $A:=\{a_1,\dots ,a_n\}$ and 
$B:=\{b_1,\dots ,b_k\}$, when the pair of families $(A,B)$ is assumed to be cyclic independent. As a direct consequence, we are able to describe the asymptotic behavior of the eigenvalues of polynomials in random matrices with discrete spectrum.

Our method has the advantage that it is general and avoids using ad-hoc combinatorial arguments.
Without going into details, let us notice that the distribution of the above polynomials only use a very small information of $b$, namely $\tau(b)$ and $\tau(b^2)$. The important observation is that polynomials may be realized as the entry (1,1) of products of matrices and thus studying matrices with cyclic entries is enough to calculate the above distributions. Once this is observed, the nature of the above formulas will be clear. 

The paper is organized as follows. In Section 2 we give the necessary preliminaries explaining the theory initiated in \cite{CHS} about cyclic monotone independence including the connection with random matrices. In Section 3 we prove the main results on polynomials in cyclic monotone families. Finally, in Section 4 we present applications to random matrices through examples.

\section{Preliminaries}

\subsection{Cyclic Monotone Independence}

\begin{defi}
A \textit{non-commutative probability space} is a pair $(\A,\tau)$ such that $\A$ is a unital $*$-algebra over $\mathbb{C}$ and $\tau:\A\to\mathbb{C}$ is a linear functional such that $\tau(1_\A)=1$.
\end{defi}
Given a non-commutative probability space $(\A,\tau)$, we say that $\tau$ is \textit{tracial} if $\tau(ab)=\tau(ba)$ for all $a,b\in \A$. All the non-commutative probability spaces considered in this paper have the property of traciality.
\par Given a separable Hilbert space $H$, we recall that a selfadjoint compact operator $a$ on $H$ is \textit{trace class} is $\sum_{i=1}^\infty |\lambda_i|<\infty$, where $\{\lambda_i\}_i$ is the set of eigenvalues of $a$. The set of trace class operators forms an ideal in the space of bounded operators on $H$, $\mathcal{B}(H)$. One way to establish the later notion in an abstract framework is given in the next definition.
\begin{defi}
Let $\A$ be a $*$-algebra over $\mathbb{C}$. We say that $\omega:\Du\to\mathbb{C}$ is a \textit{tracial weight on $\A$} if $\omega$ is a linear functional defined in a $*$-subalgebra $\Du\subset \A$ such that $\omega$ is tracial ($\omega(ab)=\omega(ba)$ for all $a,b\in \Du$), selfadjoint ($\omega(a^*)=\overline{\omega(a)}$ for all $a\in\Du$), and positive ($\omega(a^* a)\geq 0$ for all $a\in \Du$). The pair $(\A,\omega)$ is called a \textit{non-commutative measure space}.
\end{defi}

If $H$ is a separable Hilbert space, then the trace on $H$, $\Tr_H$, is an example of a tracial weight on $\mathcal{B}(H)$.
\par  We are interested in considering non-commutative probability spaces along with a tracial weight such that the random variables are trace class operators on a separable Hilbert space $H$. More precisely, we have the following definition.

\begin{defi}
Let $(\A,\tau)$ be a non-commutative probability space with tracial weight $\omega$.
\begin{enumerate}
\item Let $a_1,\ldots,a_k\in \Du$. The \textit{distribution of $(a_1,\ldots,a_k)$} is the set of moments of $a_1,\ldots,a_k$ with respect to $\omega$:
$$\left\{\omega\left(a_{i_1}^{\epsilon_1}\cdots a_{i_m}^{\epsilon_m}  \right)\,:\,m\geq 1,\,1\leq i_1,\ldots,i_m\leq k,\, \epsilon_1,\ldots,\epsilon_n\in\{1,*\} \right\}.$$
Given another $*$-algebra $\mathcal{B}$ with a tracial weight $\psi$ and $b_1,\ldots,b_k\in \operatorname{D}(\psi) $, we say that $(b_1,\ldots,b_k)$ has the \textit{same distribution} that of  $(a_1,\ldots,a_k)$ if
$$\omega\left(a_{i_1}^{\epsilon_1}\cdots a_{i_m}^{\epsilon_m}  \right) = \psi\left(b_{i_1}^{\epsilon_1}\cdots b_{i_m}^{\epsilon_m}  \right), $$for any $m\geq1$, $1\leq i_1,\ldots,i_n\leq k,\, \epsilon_1,\ldots,\epsilon_n\in\{1,*\}$.
\item  Let $a_1,\ldots,a_k\in \Du$. We say that $(a_1,\ldots,a_k)$ is \textit{trace class with respect $\omega$} if the distribution of $(a_1,\ldots,a_k)$ with respect $\omega$ is the same that of the distribution of $(x_1,\ldots,x_k)$ with respect $\Tr_H$, where $H$ is a separable Hilbert space, $x_1,\ldots,x_k$ are trace class operators on $H$, and $\Tr_H$ is the trace on $H$.
\item If $a\in \A$ is trace class with respect $\omega$, let $x$ be a trace class operator on a separable Hilbert space $H$ with the same distribution that of $a$. We define the \textit{eigenvalues of $a$} as the eigenvalues of $x$. We denote the multiset of eigenvalues of $a$ as $\EV(a)$.
\end{enumerate}
\end{defi}

The following proposition follows immediately from Corollary 2.6 in \cite{CHS} and it is fundamental in the proofs of our results.

\begin{prop}[Corollary 2.6 in \cite{CHS}] \label{Cor2.6}
Let $(\A,\omega)$ and $(\mathcal{B},\psi)$ be $*$-algebras over $\mathbb{C}$ with tracial weights. If $a\in \Du$ and $b\in\operatorname{D}(\psi)$ are selfadjoint and trace class elements such that $\omega(a^n)=\psi(b^n)$ for every $n\geq1$, then $\EV(a)=\EV(b)$.
\end{prop}

%

Motivated for the results of Shlyakhtenko in \cite{Shl} about the asymptotic behavior of the moments of products of rotationally invariant random matrices and matrices whose all entries are zero except one of them, the authors of \cite{CHS} defined an abstract notion of independence which fits in the framework of non-commutative probability spaces provided with a tracial weight.

\begin{defi}
Let $(\C,\tau)$ be a non-commutative probability space with a tracial weight $\omega$. Let $\A,\B\subset \C$ be $*$-subalgebras  such that $1_\C\in \B$. Define the $*$-ideal \begin{equation}
\operatorname{I}_{\B}(\A) := \operatorname{span}\left\{b_0a_1b_1\cdots a_nb_n \,:\, n\in \mathbb{N},\, a_1,\ldots,a_n\in\A, b_0,\ldots,b_n\in \B  \right\} 
\end{equation}
We say that the pair $(\A,\B)$ is \textit{cyclically monotonically independent} with respect to $(\omega,\tau)$ if $\operatorname{I}_\B(\A)\subset \Du$ and if for any $n\in \mathbb{N}$, $a_1,\ldots,a_n\in \A$, $b_1,\ldots,b_n\in \B$, we have
\begin{equation}
\omega(a_1b_1a_2b_2\cdots a_nb_n) = \omega(a_1a_2\cdots a_n) \tau(b_1)\tau(b_2)\cdots \tau(b_n).
\end{equation} 
\par We can also define that the pair $(\{a_1,\ldots,a_k\},\{b_1,\ldots,b_\ell\})$ is \textit{cyclically monotone} if the pair $(\operatorname{alg}\{a_1,\ldots,a_k\},\operatorname{alg}\{1_\C,b_1,\ldots,b_\ell\})$ is cyclically monotone, where $a_1,\ldots,a_k\in \Du$ and $b_1,\ldots,b_\ell\in \B$.
\end{defi}

The authors of \cite{CHS} found explicit formulas for the eigenvalues for some polynomials in cyclically monotone elements. The precise statement of their result is the following. 

\begin{teo}[Theorem 3.14 in \cite{CHS}]
\label{T314}
Let $(\A,\tau)$ be a non-commutative probability space with tracial weight $\omega$. Consider $a,a_1,\ldots,a_k\in \Du$ and $b,b_1,\ldots,b_k \in\A$  such that $(a,a_1,\ldots,a_k)$ are trace class with respect to $\omega$ and the pair   $(\{a,a_1,\ldots,a_k \},\{b,b_1,\ldots,b_k \})$ is cyclically monotone with respect to $(\omega,\tau)$. 
\begin{enumerate}
\item If $a_1,\ldots,a_k$ are selfadjoint and $B = (\tau(b_i^*b_j))_{i,j=1}^k$, then
$$\operatorname{EV}\left(\sum_{i=1}^k b_ia_ib_i^* \right) = \operatorname{EV}\left(\sqrt{B}\operatorname{diag}(a_1,\ldots,a_k) \sqrt{B}\right),$$  where $\sqrt{B}\operatorname{diag}(a_1,\ldots,a_k)\sqrt{B} \in (M_k(\mathbb{C})\otimes \mathcal{A}, \operatorname{Tr}_k\otimes\;\omega).$
\item If $b_1,\ldots,b_k$ are selfadjoint, then
$$\operatorname{EV}\left(\sum_{i=1}^k a_ib_ia_i^* \right) = \operatorname{EV}\left(\sum_{i=1}^k \tau(b_i)a_ia_i^* \right).$$
\item If  $a,b$ are selfadjoint, $p=\sqrt{\tau(b^2)}+\tau(b)$ and $q = -\sqrt{\tau(b^2)}+ \tau(b)$  then
$$\operatorname{EV}(ab+ba) = (p\operatorname{EV}(a))\sqcup (q\operatorname{EV}(a)).$$
\item If $a,b$ are selfadjoint and   $r = \sqrt{\tau(b^2)-\tau(b)^2}$, then
$$\operatorname{EV}(\operatorname{i}(ab-ba)) = (r\operatorname{EV}(a))\sqcup (-r\operatorname{EV}(a)).$$
\end{enumerate}
\end{teo}

\subsection{Asymptotic Cyclic Monotone Independence of Random Matrices}

An important application of the theory of cyclic monotone independence is the study of matrices with discrete spectrum in the limit as the size $n$ tends to infinity. 

The correct notion to state this connection is that of of \emph{ asymptotic cyclic monotone independence}, which is the analog of asymptotic freeness in this context.

\begin{defi}
\begin{enumerate}
    \item Let $\{(\A_n,\omega_n)\}_{n\geq1}$, $(\A,\omega)$ be $*$-algebras over $\mathbb{C}$ with tracial weights, for every $n\geq1$. Assume that $a_1(n),\ldots,a_k(n)\in \operatorname{D}(\omega_n)$, for every $n\geq1$. We say that  $\{(a_1(n),\ldots,a_k(n))\}_{n\geq1}$ \textit{converges in distribution with respect to $\omega$} if there exist $a_1,\ldots,a_k\in \Du$ such that
    $$\omega_n(a_{i_1}(n)^{\epsilon_{1}}\cdots a_{i_m}(n)^{\epsilon_{m}}) = \omega( a_{i_1}^{\epsilon_{1}}\cdots a_{i_m}^{\epsilon_{m}}),$$ for any $m\geq 1$, $1\leq i_1,\ldots,i_m\leq k$, $\epsilon_1,\ldots,\epsilon_m\in \{1,*\}$.
    \item Let  $\{(\A_n,\tau_n)\}_{n\geq1}$  be  non-commutative  probability  spaces such that $\A_n$ has a tracial weight $\omega_n$, for every $n\geq1$. Let $a_1(n),\ldots ,a_k(n)\in \operatorname{D}(\omega_n),$ $b_1(n),\ldots,b_\ell(n)\in \A_n$.  We say that the pair $(\{a_1(n),\ldots,a_k(n)\},\{b_1(n),\ldots,b_\ell(n)\})$ is  \textit{asymptotically  cyclically  monotone with respect to $(\omega_n,\tau_n)$} if there exist a non-commutative probability space $(\A,\tau)$ with a tracial weight $\omega$,  and elements $a_1,\ldots,a_k\in \Du, b_1,\ldots,b_\ell\in \A$ such that the pair $(\{a_1,\ldots,a_k\},\{b_1,\ldots,b_\ell\})$ is cyclically monotone, and  for any non-commutative $*$-polynomial $P(x_1,\ldots,x_k,y_1,\ldots,y_\ell)$ such that $P(0,\ldots,0,y_1,\ldots,y_\ell) = 0$, we have that $P(a_1(n),\ldots,a_k(n),b_1(n),\ldots,b_\ell(n))\in \operatorname{D}(\omega_n)$ for every $n\geq 1$ and the following limit holds:
    $$\lim_{n\to\infty} \omega_n(P(a_1(n),\ldots,a_k(n),b_1(n),\ldots,b_\ell(n))) = \omega(P(a_1,\ldots,a_k,b_1,\ldots,b_\ell)).$$
\end{enumerate}
\end{defi}

The main application of Collins Hasebe and Sakuma \cite{CHS} to Random Matrix Theory comes from the following theorem, which allows us to give an approximation of the spectrum of certain Random Matrices with discrete spectrum by using polynomials in cyclic mononotone independent elements.

\begin{teo}[Theorem 4.3 in \cite{CHS}]
Let $n\geq1$.  Let $U=U(n)$ be a $n\times n$ Haar unitary random matrix and $A_i=A_i(n)$, $B_j=B_j(n)$, $i= 1,\ldots,k$, $j= 1,\ldots,\ell$ be $n\times n$ deterministic matrices such that $((A_1,\ldots,A_k),\Tr_n)$ converges  in  distribution  to  a $k$-tuple  of  trace  class  operators  as $n\to\infty$, and $((B_1,\ldots,B_\ell),\frac{1}{n}\Tr_n)$ converges in distribution to a $\ell$-tuple of elements in a non-commutative probability space as $n\to\infty$. Then  the  pair $(\{A_1,\ldots,A_k\},\{UB_1U^*,\ldots,UB_\ell U^* \})$ is  asymptotically  cyclically  monotone  almost surely with respect to $(\Tr_n,\frac{1}{n}\Tr_n)$.
\end{teo}

\section{Main Results}

The main objective of this paper is to broaden the applicability of cyclic monotone independence by extending Theorem \ref{T314}, which computes the eigenvalue set of specific polynomials, to general selfadjoint polynomials.

 \subsection{Matrices with Cyclic Monotone Entries} 
 The first observation is that in the above formulas, the eigenvalues only depend of the elements $b_i$'s through their expectations. Then, it is natural to ask if in general we can replace the elements $b_i$'s in the polynomials by their expectation. As one see from parts 1) and 2) this step is not obvious, but it can done as long as we consider matrices, in the correct order. This is the content of our following result.

\begin{prop}
\label{p1}
Let $(\A,\tau)$ be a non-commutative probability space with a tracial weight $\omega$. Consider $A_p  = \left(a_{ij}^{(p)}\right)_{i,j}^n \in M_n(\operatorname{D}(\omega))$ and $B_q =  \left(b_{ij}^{(q)}\right)_{i,j}^n \in M_n(\A)$ for $p,q=1,\ldots,k$. Assume that $\left(a_{ij}^{(p)}\right)_{i,j=1,p=1,\ldots,k}^n$ are trace class with respect to $\omega$ , and the pair
$$\left( \left\{a_{i,j}^{(p)}\,:\,p=1,\ldots,k,\, i,j=1,\ldots,n  \right\}, \left\{b_{i,j}^{(q)}\,:\,q=1,\ldots,k,\, i,j=1,\ldots,n  \right\} \right)  $$
is cyclically monotone independent with respect to $(\omega,\tau)$. Then
\begin{equation}
\Tr_n\otimes\; \omega\left( A_1B_1A_2B_2\cdots A_kB_k  \right) = \Tr_n\otimes\;\omega \left(A_1B_1^\prime A_2B_2^\prime\cdots A_k B_k^\prime \right),
\end{equation}
where for each $p=1,\ldots,k$ we have that
$$ B_p^\prime = \operatorname{id}_n\otimes \;\tau(B_p) = \left( \tau\left(b_{ij}^{(p)}\right)_{i,j=1}^n\right)\in M_n(\mathbb{C}).$$
\end{prop}
 \allowdisplaybreaks
\begin{proof}
Using linearity and cyclic property of $\Tr_n$ and $\omega$ we have that\\
\begin{eqnarray*}
\Tr_n\otimes\; \omega\left( A_1B_1A_2B_2\cdots A_kB_k  \right) &=& \sum_{\substack{i_1,i_2\ldots,i_k=1\\j_1,j_2\ldots,j_k=1}}^n \omega\left( a_{i_1j_i1}^{(1)} b_{j_1i_2}^{(1)}a_{i_2j_2}^{(2)}b_{j_2i_3}^{(2)}\cdots a_{i_kj_k}^{(k)}b_{j_ki_1}^{(k)}\right)
\\ &=&\sum_{\substack{i_1,i_2\ldots,i_k=1\\j_1,j_2\ldots,j_k=1}}^n \omega\left( a_{i_1j_i1}^{(1)} a_{i_2j_2}^{(2)}\cdots a_{i_kj_k}^{(k)}\right)\tau\left( b_{j_1i_2}^{(1)}\right) \tau\left(b_{j_2i_3}^{(2)}\right)\cdots\tau\left( b_{j_ki_1}^{(k)} \right)
\\ &=& \sum_{\substack{i_1,i_2\ldots,i_k=1\\j_1,j_2\ldots,j_k=1}}^n \omega\left( a_{i_1j_i1}^{(1)} \tau\left( b_{j_1i_2}^{(1)}\right)a_{i_2j_2}^{(2)}\tau\left(b_{j_2i_3}^{(2)}\right)\cdots a_{i_kj_k}^{(k)}\tau\left(b_{j_ki_1}^{(k)}\right)\right)
\\ &=& \Tr_n\otimes\;\omega \left(A_1B_1^\prime A_2B_2^\prime\cdots A_k B_k^\prime \right).
\end{eqnarray*}
\end{proof}

Since a power of a matrix $A_1B_1\cdots A_kB_k$ has the same form, we have the following consequence.
\begin{cor}
With the assumptions and notation of Proposition \ref{p1}, for any $m\geq1$ we have that
\begin{equation}
\Tr_n\otimes\; \omega\left(\left( A_1B_1A_2B_2\cdots A_kB_k\right)^m  \right) = \Tr_n\otimes\;\omega \left(\left(A_1B_1^\prime A_2B_2^\prime\cdots A_k B_k^\prime \right)^m\right),
\end{equation}
i.e., the moments of $A_1B_1A_2B_2\cdots A_kB_k$ and $A_1B_1^\prime A_2B_2^\prime\cdots A_kB_k^\prime$ with respect $\Tr_n\otimes\;\omega$  are the same.
\end{cor}


By applying Proposition \ref{Cor2.6}, we arrive at the following result.

\begin{teo} \label{maintheorem}Let $(\A,\tau)$ be a non-commutative probability space with a tracial weight $\omega$. Consider $A_p  = \left(a_{ij}^{(p)}\right)_{i,j}^n \in M_n(\operatorname{D}(\omega))$ and $B_q =  \left(b_{ij}^{(q)}\right)_{i,j}^n \in M_n(\A)$ for $p=1,\ldots,k$, $q=0,\ldots,k$. Assume that $\left(a_{ij}^{(p)}\right)_{i,j=1,p=1,\ldots,k}^n$ are trace class with respect to $\omega$ , and the pair
$$\left( \left\{a_{i,j}^{(p)}\,:\,p=1,\ldots,k,\, i,j=1,\ldots,n  \right\}, \left\{b_{i,j}^{(q)}\,:\,q=0,\ldots,k,\, i,j=1,\ldots,n  \right\} \right)  $$
is cyclically monotone independent with respect to $(\omega,\tau)$. Then
$$\EV(B_0A_1B_1\cdots A_kB_k) = \EV(A_1B_1^\prime \cdots A_k (B_kB_0)^\prime),$$
where $B_0^\prime,\ldots,B_k^\prime$ are defined as in Proposition \ref{p1}.
\end{teo}

The above theorem reduces the calculation of the eigenvalues of $B_0A_1B_1\cdots A_kB_k$ to the eigenvalues of $A_1B_1^\prime \cdots A_k (B_kB_0)^\prime$. Thus we only deal with matrices in the elements $a_i$'s and constant matrices $B_i^\prime$'s. As an application, we can give a proof of Theorem \ref{T314}.

\begin{proof}[Proof of Theorem \ref{T314}]
$1)\;$ Assume that $a_1,\ldots,a_k$ are selfadjoint and consider the matrix $B= (\tau(b_i^*b_j))_{i,j=1}^k$. Define the matrices
\begin{equation*}
A_1 = \begin{pmatrix}
a_1 & 0 &\cdots& 0\\
0 & a_2& \cdots &0\\
\vdots&\vdots &\ddots &\vdots\\
0 & 0&\cdots & a_k
\end{pmatrix}, 
\qquad B_ 0 =  \begin{pmatrix}
b_1 & b_2 &\cdots& b_k\\
0 & 0& \cdots &0\\
\vdots&\vdots &\ddots &\vdots\\
0 & 0&\cdots & 0
\end{pmatrix}.
\end{equation*}
and notice that $$B_0A_1B_0^* =   \begin{pmatrix}
\displaystyle \sum_{i=1}^k b_ia_ib_i^* &  0 &\cdots& 0\\
0 & 0& \cdots &0\\
\vdots&\vdots &\ddots &\vdots\\
0 & 0&\cdots & 0
\end{pmatrix},$$
which is a selfadjoint element in $M_k(\mathbb{C})\otimes \A$, with the same moments as $\sum_{i=1}^k b_ia_ib_i^*$. That is, for $m\geq1$, we have that 
$$\Tr\otimes \;\omega( (B_0A_1B_0^*)^m ) = \omega\left(\left( \sum_{i=1}^k b_ia_ib_i^*\right)^m \right).$$

We want to understand the eigenvalues of $B_0A_1B_0^*$. By traciality,  with respect to $\Tr_k\otimes\,\omega$, the moments of $B_0A_1B_0^*$, are the same as the moments of $A_1B_0^*B_0$ and thus, due to Corollary \ref{Cor2.6},  $B_0A_1B_0^*$ and $A_1B_0^*B_0$, must have the same eigenvalues. 

Finally, by Theorem \ref{maintheorem}, $A_1B_0^*B_0$ has the same eigenvalues as $AB$. Since $B$ is positive definite, the eigenvalues of $B_0A_1B_0^*$  are the same as the eigenvalues of $\sqrt{B}A_1\sqrt{B}$, as desired.

$2)\;$ Using the same idea of above, if we define
\begin{eqnarray*}
A_1 =   \begin{pmatrix}
a_1 &  a_2 &\cdots& a_k\\
0 & 0& \cdots &0\\
\vdots&\vdots &\ddots &\vdots\\
0 & 0&\cdots & 0
\end{pmatrix},\qquad
B_1 = \operatorname{diag}(b_1,\ldots,b_k),
\end{eqnarray*}
we have that $$A_1B_1A_1^* = \operatorname{diag}\left(\sum_{i=1}^k a_ib_ia_i^*,0,\ldots,0 \right)$$
which is selfadjoint. By Theorem \ref{maintheorem}, if $B_1^\prime = \operatorname{diag}(\tau(b_1),\tau(b_2),\ldots,\tau(b_k))$ we have that for $EV[A_1BA_1^*]$ any $m\geq1$
\begin{eqnarray*}
\omega\left( \left( 	\sum_{i=1}^k a_ib_ia_i^*\right)^m\right) &=& \Tr_k\otimes\;\omega\left( (A_1B_1A_1^*)^m\right)
\\ &=& \Tr_k\otimes\;\omega\left((A_1B_1^\prime A_1^*)^m \right)
\\ &=& \omega\left(\left(\sum_{i=1}^k a_i\tau(b_i)a_i^* \right)^m \right).
\end{eqnarray*}
Hence
$$\EV\left( \sum_{i=1}^k a_ib_ia_i^* \right) = \EV\left( \sum_{i=1}^k \tau(b_i)a_ia_i^*\right).$$
$3)\;$ Assume that $a,b$ are selfadjoint. Define the matrices
$$B_0 = \begin{pmatrix}
1&b\\ 0 & 0
\end{pmatrix},\quad A_1 = \begin{pmatrix}
a&0\\0&a
\end{pmatrix},\quad B_1 = \begin{pmatrix}
b&0\\1&0
\end{pmatrix}.$$
Then we have that $$B_0A_1B_1 = \begin{pmatrix}
ab+ba&0\\0&0
\end{pmatrix}$$which is selfadjoint. Then by Proposition \ref{p1}, for any $m\geq1$ we have that
\begin{eqnarray*}
\omega((ab+ba)^m) &=& \Tr_2\otimes\;\omega((B_0A_1B_1)^m)
\\ &=& \Tr_2\otimes\;\omega ((A_1B_1B_0)^m)
\\ &=& \Tr_2\otimes \;\omega \left( \left( \begin{pmatrix}
a&0\\0&a
\end{pmatrix} \begin{pmatrix}
\tau(b)&\tau(b^2)\\1&\tau(b)
\end{pmatrix}\right)^m\right)
\\ &=& \Tr_2\otimes\;\omega \left( a^m \begin{pmatrix}
\tau(b)&\tau(b^2)\\1&\tau(b)
\end{pmatrix}^m \right),
\end{eqnarray*}
If we have a matrix $ \begin{pmatrix}
x&y\\1&x
\end{pmatrix}$, by diagonalizing we have that $$\Tr_2\left( \begin{pmatrix}
x&y\\1&x
\end{pmatrix}^m\right) = (x+\sqrt{y})^m + (x-\sqrt{y})^m.$$ Then
\begin{eqnarray*}
\omega((ab+ba)^m) &=&\Tr_2\otimes\omega \left( a^m \begin{pmatrix}
\tau(b)&\tau(b^2)\\1&\tau(b)
\end{pmatrix}^m \right) \\ &=& \Tr_2\otimes \omega \left( a^m \begin{pmatrix}
p^m&0\\0&q^m
\end{pmatrix} \right) \\ &=& \omega \left( a^m\left(\tau(b)+\sqrt{\tau(b^2)}\right)^m + a^m\left(\tau(b)-\sqrt{\tau(b)^2}\right)^m\right)\\ &=& \omega\left((pa)^m +(qa)^m\right),
\end{eqnarray*}

where $p$ and $q$ are defined as in the statement of Theorem \ref{T314}. We conclude that
 $$\EV(ab+ba) = p\EV(a) \sqcup q\EV(a).$$
 $4)\;$ Proceeding in an analogous way of 3), defining the matrices
 $$B_0 = \begin{pmatrix}
 \operatorname{i}& -\operatorname{i}b\\ 0 &0
\end{pmatrix},\quad A_1= \begin{pmatrix}
a&0\\0&a
\end{pmatrix},\quad B_1= \begin{pmatrix}
b&0\\1&0
\end{pmatrix}$$
we have that
$$B_0A_1B_1 = \begin{pmatrix}
\operatorname{i}(ab-ba)&0\\0&0
\end{pmatrix}$$which is selfadjoint. Then for any $m\geq1$ we have that
\begin{eqnarray*}
\omega(((\operatorname{i}(ab+ba))^m) &=& \Tr_2\otimes\;\omega((B_0A_1B_1)^m)
\\ &=& \Tr_2\otimes\;\omega ((A_1B_1B_0)^m)
\\ &=& \Tr_2\otimes \;\omega \left( \left( \begin{pmatrix}
a&0\\0&a
\end{pmatrix} \begin{pmatrix}
\operatorname{i}\tau(b)&-\operatorname{i}\tau(b^2)\\\operatorname{i}&-\operatorname{i}\tau(b)
\end{pmatrix}\right)^m\right)
\\ &=& \Tr_2\otimes\;\omega \left( a^m\begin{pmatrix}
\operatorname{i}\tau(b)&-\operatorname{i}\tau(b^2)\\\operatorname{i}&-\operatorname{i}\tau(b)
\end{pmatrix}^m \right),
\end{eqnarray*}
By diagonalizing, we have that
$$\Tr_2\left( \begin{pmatrix}
x&y\\1&-x
\end{pmatrix}^m\right) = \left(\sqrt{y+x^2}\right)^m + \left(-\sqrt{y+x^2}\right)^m.$$
Hence
\begin{eqnarray*}
\omega((\operatorname{i}(ab-ba))^m) &=&  \Tr_2\otimes\;\omega \left( a^m\begin{pmatrix}
\operatorname{i}\tau(b)&-\operatorname{i}\tau(b^2)\\\operatorname{i}&-\operatorname{i}\tau(b)
\end{pmatrix}^m \right)\\ &=&  \Tr_2\otimes\;\omega \left( a^m\begin{pmatrix}
r^m&0\\0&(-r)^m
\end{pmatrix} \right) \\  &=& \omega \left(a^m \left(\sqrt{\tau(b^2)- \tau(b)^2}\right)^m + a^m  \left(-\sqrt{\tau(b^2)- \tau(b)^2}\right)^m\right) \\ &=& \omega\left( (ra)^m + (-ra)^m\right),
\end{eqnarray*}
where $r$ is defined as in the statement of Theorem \ref{T314}. We conclude that
$$\EV(\operatorname{i}(ab-ba)) = r\EV(a)\sqcup (-r)\EV(a).$$
\end{proof}

Now, in parts 3) and 4) of Theorem \ref{T314} we only one consider one element $a$. In this cases, it was possible to obtain explicit formulas from the fact that $a$ commutes with itself and then we can compute the trace of powers of a matrix by adding powers of the eigenvalues. The result can be expressed as the disjoint union of the eigenvalues of $\lambda_i a$, where the $\lambda_i$'s are the eigenvalues of $\operatorname{id}\otimes \tau \;(B)$. One can generalize this as in the following result.

\begin{prop}
\label{p6}
Let $(\A,\omega,\tau)$ be a non-commutative probability space with a tracial weight $\omega$. Consider $a\in \Du$ and $b_1,\ldots,b_k,c_1,\ldots,c_k \in \A$ such that $a$ is trace class with respect to $\omega$ and $(\{a\},\{b_1,c_1,\ldots,b_k,c_k\})$ is cyclically monotone with respect to $(\omega,\tau)$. Assume that $a,b_1,c_1,\ldots,b_k,c_k$ are selfadjoint and $B^\prime = (\tau(c_ib_j))_{i,j=1}^k$. If $\lambda_1,\ldots,\lambda_k$ are the $k$ eigenvalues of $B^\prime$ counting multiplicity, then
\begin{equation}
\EV\left(\sum_{i=1}^k b_iac_i \right) = \bigsqcup_{i=1}^k \EV(\lambda_i a).
\end{equation}
\end{prop}
\begin{proof}
As in the proof of Theorem \ref{T314}, define the matrices in $M_k(\A)$
\begin{equation*}
B = \begin{pmatrix}
b_1 & b_2 &\cdots& b_k\\
0 & 0& \cdots &0\\
\vdots&\vdots &\ddots &\vdots\\
0 & 0&\cdots & 0
\end{pmatrix}, 
\quad C =  \begin{pmatrix}
c_1 & 0 &\cdots& 0\\
c_2 & 0& \cdots &0\\
\vdots&\vdots &\ddots &\vdots\\
c_k & 0&\cdots & 0
\end{pmatrix},\quad A = \operatorname{diag}(a,\ldots,a).
\end{equation*}
Then $BAC = \operatorname{diag}\left(\sum_{i=1}^k b_ia c_k,0,\ldots,0 \right)$ which is selfadjoint. Proceeding as in the latter proof we have that for $m\geq1$
\begin{eqnarray*}
\omega\left( \left(\sum_{i=1}^k b_iac_i \right)^m\right) &=& \Tr_k\otimes\,\omega ((BAC)^m)
\\ &=& \Tr_k\otimes\, \omega ((ACB)^m)
\\ &=& \Tr_k\otimes\, \omega \left(  (aB^\prime)^m\right)
\\ &=& \omega\left(\sum_{i=1}^k (\lambda_ia)^m \right),
\end{eqnarray*}
where we use that $\Tr_k(X^m)$ is the sum of the $m$-powers of the eigenvalues of $X$. We conclude then that
\begin{equation*}
\EV\left(\sum_{i=1}^k b_iac_i \right) = \bigsqcup_{i=1}^k \lambda_i\EV(a).
\end{equation*}
\end{proof}

\subsection{Conjugation Respects Cyclic Monotone Independence}

One can asks if is possible to form new cyclic monotone elements from given ones. An answer for this question is provided in the following proposition. Combined with Theorem \ref{T314}, it also allows to get a formula for a new kind of polynomials.

\begin{prop}
\label{p10}
Let $(\A,\omega,\tau) $ be a non-commutative probability space with  tracial weight $\omega$. Let $a_1,\ldots,a_k\in \Du$ and $b_1,\ldots,b_k,c_1,\ldots,c_k \in \A$. If $(\{a_1,\ldots,a_k\},\{b_1,c_1,\ldots,b_k,c_k\})$ is cyclically monotone, then $(\{a_1c_1a_1^*,\ldots,a_kc_ka_k^*\},\{b_1,c_1,\ldots,b_k,c_k\})$ is cyclically monotone.
\end{prop}
\begin{proof}
We have to show that if $x_1,\ldots,x_n \in \operatorname{alg}(\{a_1c_1a_1^*,\ldots,a_kc_ka_k^*\})$ and $y_1,\ldots,y_n \in \operatorname{alg}(\{1,b_1,c_1,\ldots,b_k,c_k\})$, then 
$$\omega(x_1y_1\cdots x_ny_n) = \omega(x_1\cdots x_n)\tau(y_1)\cdots \tau(y_n).$$
For notational convenience, we will prove the result for the case $n=2$. The general case is done in a similar way. Consider the elements
\begin{eqnarray*}
x_1 &=& (a_{i_1}c_{i_1}a_{i_1}^*)\cdots(a_{i_r}c_{i_r}a_{i_r}^*),
\\ x_2 &=& (a_{j_1}c_{j_1}a_{j_1}^*)\cdots(a_{j_s}c_{j_s}a_{j_s}^*),
\end{eqnarray*}
for some $1\leq i_1,\ldots,i_r,j_1,\ldots,j_s\leq k$. If $y_1,y_2\in \operatorname{alg}(\{1,b_1,c_1,\ldots,b_k,c_k\})$, by cyclic monotone independence, we have that
\begin{eqnarray*}
\omega(x_iy_1x_2y_2) &=& \omega \left( (a_{i_1}c_{i_1}a_{i_1}^*)\cdots(a_{i_r}c_{i_r}a_{i_r}^*)y_1  (a_{j_1}c_{j_1}a_{j_1}^*)\cdots(a_{j_s}c_{j_s}a_{j_s}^*)y_2\right)
\\ &=& \omega(a_{i_1}a_{i_1}^*\cdots a_{i_r}a_{i_r}^*a_{j_1}a_{j_1}^*\cdots a_{j_s}a_{j_s}^*)\tau(c_{i_1})\cdots \tau(c_{i_r})\tau(y_1)\tau(c_{j_1})\cdots\tau(c_{j_s})\tau(y_2)
\end{eqnarray*}
On the other hand
\begin{eqnarray*}
\omega(x_1x_2)\tau(y_1)\tau(y_2) &=& \omega\left((a_{i_1}c_{i_1}a_{i_1}^*)\cdots(a_{i_r}c_{i_r}a_{i_r}^*)(a_{j_1}c_{j_1}a_{j_1}^*)\cdots(a_{j_s}c_{j_s}a_{j_s}^*) \right)\tau(y_1)\tau(y_2)
\\ &=& \omega(a_{i_1}a_{i_1}^*\cdots a_{i_r}a_{i_r}^*a_{j_1}a_{j_1}^*\cdots a_{j_s}a_{j_s}^*)\left(\prod_{\ell=1}^r\tau(c_{i_\ell})\right)\left(\prod_{\ell=1}^s \tau(c_{j_\ell})\right)\tau(y_1)\tau(y_2)
\end{eqnarray*}
We finish the proof by comparing the above equations.
\end{proof}

As we stated before, last proposition allows to get a formula for a generalization of the first part of Theorem \ref{T314}.

\begin{cor}
Let $(\A,\tau)$ be a non-commutative probability space with tracial weight $\omega$. Consider $a_1,\ldots,a_k\in \Du$ and $b_1,\ldots,b_k,c_1,\ldots,c_k \in\A$  such that $(a_1,\ldots,a_k)$ are trace class with respect to $\omega$ and  $(\{a_1,\ldots,a_k \},\{b_1,c_1,\ldots,b_k,c_k\})$ is cyclically monotone with respect to $(\omega,\tau)$. If $c_1,\ldots,c_k$ are selfadjoint and $B = ((\tau(b_i^*b_j)))_{i,j=1}^k\in M_k(\mathbb{C})$, then
\begin{eqnarray*}
\EV \left(\sum_{i=1}^k b_ia_ic_ia_i^*b_i^* \right) &=& \EV\left(\sqrt{B}\operatorname{diag}(d_1,\ldots,d_k)\sqrt{B}\right) \\ &=& \EV\left(\sqrt{B}\operatorname{diag}(\tau(c_1)a_1a_1^*,\ldots,\tau(c_k)a_ka_k^*)\sqrt{B} \right),
\end{eqnarray*}
where the elements $d_i = a_ic_ia_i^*\in \Du$ for $i=1,\ldots,k$, and $\sqrt{B}\operatorname{diag}(d_1,\ldots,d_k)\sqrt{B}$ and $\sqrt{B}\operatorname{diag}(\tau(c_1)a_1a_1^*,\ldots,\tau(c_k)a_ka_k^*)\sqrt{B}$ are selfadjoint elements in $(M_k(\mathbb{C})\otimes \A,\Tr_k\otimes\;\omega)$.
\end{cor}
\begin{proof}
By Proposition \ref{p10}, we have that $(\{d_1,\ldots,d_k\},\{b_1,\ldots,b_k\})$ is cyclically monotone. Since $c_i$ es selfadjoint, then $d_i$ is also selfadjoint, for $i=1,\ldots,k$. We obtain the first equality by applying Theorem \ref{T314}. The second equality follows from using Proposition \ref{p1} and the same ideas that in the proof of Theorem \ref{T314}.
\end{proof}

\subsection{General Polynomials}

In principle, by using the proof of Theorem \ref{T314} could be applied to any selfadjoint $*$-polynomial that can be written as the entry (1,1) of a product of matrices $A_1B_1\cdots A_kB_k$ as in Proposition \ref{p1}, and the rest of the entries are zero. However, the same trick is no longer possible in some polynomials where the number of elements $a_i$ is not the same on each monomial. For instance, consider the polynomial $a +babab$, where $(\{a\},\{b\})$ is cyclically monotone. If we would like to write this polynomial as a product of matrices as in the above proofs, we would have to do the following factorization
\begin{equation*}
\begin{pmatrix}
1 & b\\0 & 0
\end{pmatrix}
\begin{pmatrix}
a & 0\\0 & a
\end{pmatrix}
\begin{pmatrix}
1 & 0\\0 & b
\end{pmatrix}
\begin{pmatrix}
1 & 0\\0 & a
\end{pmatrix}
\begin{pmatrix}
1 & 0\\b & 0
\end{pmatrix} =  \begin{pmatrix}
a+ babab & 0\\0 & 0
\end{pmatrix}.
\end{equation*}
Now, if we look at the fourth matrix in the product, we notice that 1 is in the entry (1,1), then we would need that $1\in \mathcal{A}=\operatorname{alg}(a)$, so that this matrix belongs to $M_2(\mathcal{A})$. However, if we have a pair $(\A,\B)$ cyclically monotone in $\C$ such that $\A$ are trace class on a infinite dimensional Hilbert space, then $1\not\in \A$ since the identity is not compact (and hence it is not trace class).

Thus, we have some restrictions on when can we use the above procedure. In order to consider general polynomials we will use another observation: in many cases we can find simpler polynomials which have the same eigenvalues. in order to do this we need to  consider the joint distributions. More precisely, we have the following.

\begin{defi}
\begin{itemize}
\item Let $(\C,\omega)$ be a non-commutative measure space. Consider $a_1,\ldots,a_k\in \Du$, $b_1,\ldots,b_\ell \in \C$, and the subalgebras $\A = \operatorname{alg}(a_1,\ldots,a_k)$ and $\B = \operatorname{alg}(1_\C,b_1,\ldots,b_\ell)$. If for $n\geq1$, $x_1,\ldots,x_n\in \A$ and $y_0,\ldots,y_n\in \B$, we have that $y_0x_1y_1\cdots x_ny_n\in \Du$, we define the \textit{joint distribution} of $((a_1,\ldots,a_k),(b_1,\ldots,b_\ell))$ as the set of mixed moments
$$\{\omega(y_0x_1y_1\cdots x_ny_n)\;:\;n\in\mathbb{N},\,x_1,\ldots,x_n\in \A,\,y_0,\ldots,y_n\in \B\}.$$
\item Let $(\C^\prime,\omega)$ and $(\C^{\prime\prime},\xi)$ be non-commutative measure space. Consider the elements $a_1,\ldots,a_k\in \Du$, $b_1,\ldots,b_\ell \in \C$, $c_1,\ldots,c_k\in \operatorname{D}(\xi)$, and $d_1,\ldots,d_\ell \in \mathcal{D}$. Define $\A = \operatorname{alg}(a_1,\ldots,a_k)$, $\B = \operatorname{alg}(1_{\C^\prime},b_1,\ldots,b_\ell)$, $\C = \operatorname{alg}(c_1,\ldots,c_k)$, and $\mathcal{D} = \operatorname{alg}(1_{\C^{\prime\prime}},d_1,\ldots,d_\ell)$.  We say that $((a_1,\ldots,a_k),(b_1,\ldots,b_\ell))$ and $((c_1,\ldots,c_k),(d_1,\ldots,d_\ell))$ \textit{has the same joint distribution} if
$$\omega(y_0x_1y_1\cdots x_ny_n) = \xi(\Phi(y_0x_1y_1\cdots x_ny_n))$$for any $n\geq1$, $x_1,\ldots,x_n\in \operatorname{alg}(a_1,\ldots,a_k)$, and $y_0,\ldots,y_n\in \operatorname{alg}(1_\C,b_1,\ldots,b_\ell)$, and $\Phi : \operatorname{I}_\A(\B) \to \operatorname{I}_\C(\mathcal{D})$ is a unital algebra isomorphism such that $\Phi(a_i) = c_i$ for any $i=1,\ldots,k$, and $\Phi(b_j)=d_j$, for any $j=1,\ldots,\ell$.
\end{itemize}
\end{defi}

 Now we can pursue the above idea of replacing the elements $b$'s by its mean, $\tau(b)$, then reduce to a polynomial as the ones treated in Theorem \ref{T314} in order to compute the eigenvalues.

\begin{teo}\label{chido}
Let $(\A,\omega,\tau)$ be a non-commutative probability space with tracial weight $\omega$. Consider $a,c,a_1,\ldots,a_k\in \Du$ and $b,b_1,\ldots,b_l \in\A$  such that $(a,a_1,\ldots,a_k)$ are trace class with respect to $\omega$ and the pair   $(\{a,c,a_1,\ldots,a_k \},\{b,b_1,\ldots,b_k \})$ is cyclically monotone with respect to $(\omega,\tau)$. Then $((abc,a_1,\ldots,a_k),( b,b_1,\ldots,b_l))$ and $((\tau(b)ac,a_1,\ldots,a_k),( b,b_1,\ldots,b_l))$ have the same joint distribution. In particular, $(\{abc,a_1,\ldots,a_k\},\{b,b_1,\ldots,b_\ell\})$ is cyclically monotone.
\end{teo}

\begin{proof}
It is enough to prove the case when $k=\ell$ and 
\begin{equation}
\label{eq6}
    \omega\left((abc)ba_1b_1\cdots a_kb_k \right) = \omega\left( (\tau(b)ac)ba_1b_1\cdots a_kb_k \right).
\end{equation}
Using the cyclic monotone independence formula we easily get that
\begin{eqnarray*}
\omega\left((abc)ba_1b_1\cdots a_kb_k \right)  &=& \omega(aca_1\cdots a_k)\tau(b)^2\tau(b_1)\cdots \tau(b_k)
\\ &=& \omega((\tau(b)ac)a_1\cdots a_k)\tau(b)\tau(b_1)\cdots \tau(b_k)
\\ &=& \omega\left( (\tau(b)ac)ba_1b_1\cdots a_kb_k \right).
\end{eqnarray*}
The cyclic monotone independence of the pair $(\{abc,a_1,\ldots,a_k\},\{b,b_1,\ldots,b_k\})$ follows from the first equality above and and the fact that 
\begin{eqnarray*}
\omega(aca_1\cdots a_k)\tau(b) &=& \omega(aca_1\cdots a_k)\tau(b)\tau(1)^{k+1}
\\ &=& \omega((abc)\cdot 1_\A\cdot a_1 \cdot 1_\A \cdots a_k \cdot 1_\A)
\\ &=& \omega((abc)a_1\cdots a_k).
\end{eqnarray*}

\end{proof}

The basic idea of the above theorem is that if $(\A,\B)$ is cyclic monotone, whenever an element $b\in \B$ is multiplied by the left and the right by elements $a,b\in \A$, for instance $abc$, we can just take instead $\tau(b)ac$ in order to compute moments respect to $\omega$. We describe the method applied to the above polynomial.

\begin{ejem}\label{a+ababa}
Given $a,b\in (\A,\omega,\tau)$ selfadjoint elements such that $a\in \Du$, consider the polynomial $a +babab$, where $(a,b)$ is cyclically monotone. By Proposition \ref{p10} we have that $(\{a,aba\},\{b\})$ is cyclically monotone.  Now, from the above theorem we have that $(a,aba)$ has the same distribution that of  $(a,\tau(b)a^2)$. Thus $\EV(a +babab)=\EV(a +b(\tau(b)a^2)b)$. Finally,  we can use part i) of Theorem \ref{T314} in order to find that the set of eigenvalues of $a+babab$ is the same that of $\EV\left(\sqrt{B}\operatorname{diag}(a,\tau(b)a^2)\sqrt{B} \right),$ where $$B = \begin{pmatrix} 1&\tau(b)\\ \tau(b)& \tau(b^2) \end{pmatrix}.$$
In Example \ref{a+ababa matrix}, we consider how to use this calculation to give the asymptotic behavior of the eigenvalues of a random matrix related to this polynomial.
\end{ejem}

\section{Applications to Random Matrices}
In an analogous way as Theorem \ref{T314} was combined with asymptotic cyclic monotone independence in order to get the limiting set of eigenvalues of some polynomials of random matrices, we can combine our Proposition \ref{p1} and Proposition \ref{p6} with the asymptotic cyclic monotone independence in order to get a new formula. The precise statement corresponding to Proposition \ref{p6} is the following.

\begin{prop}
\label{P7}
Let $n\in\mathbb{N}$. Let $U = U(n)$ be an $n\times n$ Haar unitary random matrix and $A  = A(n)$, $B_i = B_i(n)$, $C_j = C_j(n)$, $i,j=1,\ldots,k$, be $n\times n$ deterministic matrices such that
\begin{enumerate}
\item $A$ is Hermitian and $(A,\Tr_n)$ converges in distribution to a trace class operator operator $(a,\Tr_H)$ as $n\to\infty$,
\item $((B_1,C_1\ldots,B_k,C_k), \tr_n)$ converges in distribution to a $2k$-tuple of elements in a non-commutative probability space as $n\to\infty$.
\end{enumerate}
Under the assumption 2), let $\beta_{ij} = \lim_{n\to\infty} \tr_n(C_iB_j)$ and $B^\prime = (\beta_{ij})_{i,j=1}^k$. Let $\lambda_1,\ldots,\lambda_k$ be the $k$ eigenvalues of $B^\prime$ counting multiplicity. If 
$$ \sum_{i=1}^k (UB_iU^*)A(UC_iU^*)$$ is Hermitian, then
\begin{equation}
\lim_{n\to\infty} \EV \left( \sum_{i=1}^k (UB_iU^*)A(UC_iU^*)\right) = \bigsqcup_{i=1}^k \lambda_i\EV(a).
\end{equation}
\end{prop}

As the results of \cite{CHS}, we can consider only the case when the matrices $A_i, B_i$ and $C_i$ are deterministic, since by conditioning these matrices to be constant matrices, we can get the corresponding results for the case when $A_i, B_i$ and $C_i$ are random matrices independent of the matrix $U$ and $A$ and $(B_1,C_1,\ldots,B_k,C_k)$ converge in distribution to deterministic elements. The same can be done also for Proposition \ref{p1}.

\begin{ejem}
We illustrate the random asymptotic version in random matrices of Proposition \ref{p1}. Let $n=300$. Let $B_1,B_2$ and $B_3$ be $n\times n$ independent selfadjoint GUE random matrices, $D=\operatorname{diag}(2^0,2^{-1},\ldots,2^{-n+1})$, and $U_1,U_2$ be independent Haar unitary random matrices. We take $A_1 = D, A_2 = U_1DU_1^*$ and $A_3 = U_2DU_2^*$. Define the block matrices
\begin{equation}
A = \begin{pmatrix}
A_1 & A_2 \\ A_2^* &A_3
\end{pmatrix},\quad B = \begin{pmatrix}
B_1^2 & B_2^2 \\ B_2^2& B_3^2
\end{pmatrix}.
\end{equation}
We show a realization of the eigenvalues of $BAB$ and we compare them with the eigenvalues of $A^\prime B^\prime$, where $A^\prime$ is the limit operator of $A$ and
\begin{eqnarray*}
B^\prime &=& \lim_{n\to\infty}(\operatorname{Id}\otimes \tr_n) (B^2)\\ &=&  \begin{pmatrix}\displaystyle 
\lim_{n\to\infty} \tr_n(B_1^4+B_2^4) & \displaystyle \lim_{n\to\infty}\tr_n(B_1^2B_2^2 + B_2^2B_3^2) \\ \displaystyle \lim_{n\to\infty}\tr_n(B_2^2B_1^2 + B_3^2B_2^2) & \displaystyle \lim_{n\to\infty} \tr_n(B_2^4+B_3^4)
\end{pmatrix}  \\ &=& \begin{pmatrix}
4 &2\\2&4
\end{pmatrix}.
\end{eqnarray*}
In the last equality, we use the fact that the $B_1,B_2$ and $B_3$ asymptotically behaves as a free semicircular family. A plot of the first 15 eigenvalues of $BAB$ and $A^\prime B^\prime$, and a comparison of the first three moments are provided.
\begin{center}{
\begin{figure}[H]
\centering\includegraphics[scale=0.33]{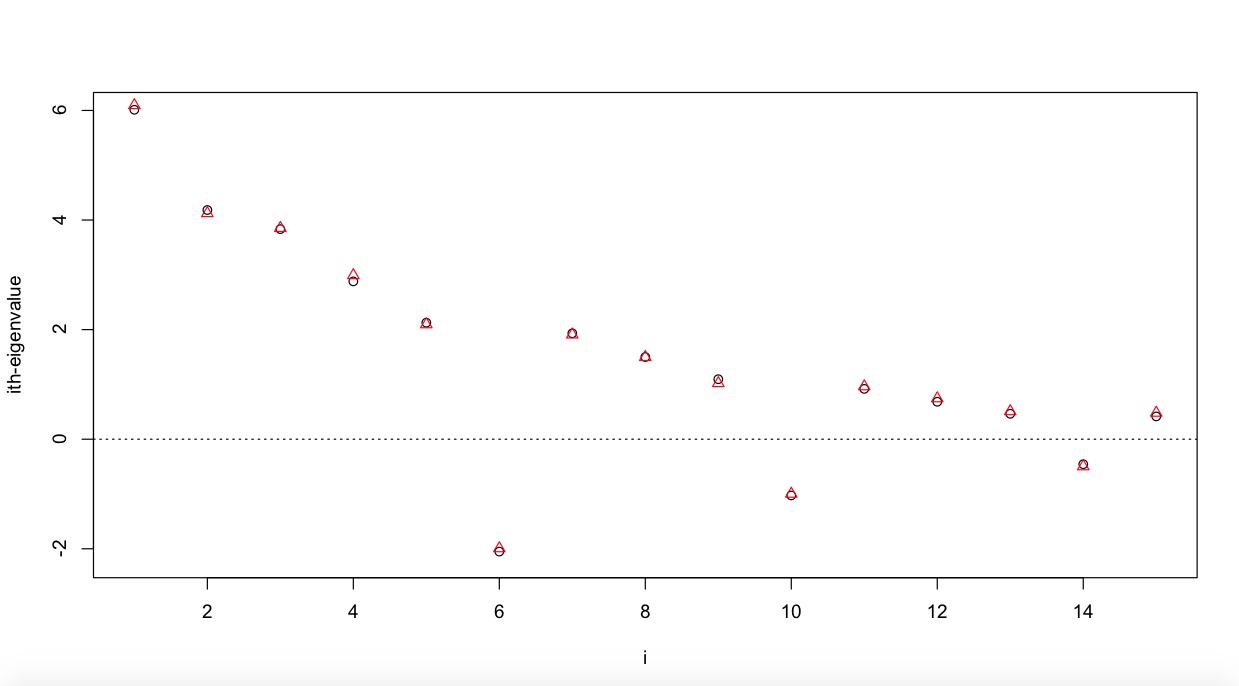}
\caption{Comparison between eigenvalues of $BAB$ (black circle) and $A^\prime B^\prime$ (red triangle).}
\end{figure}}
\end{center}

\begin{table}[H]
\begin{center}
\begin{tabular}{|c|c|c|}
\hline
 $k$& $\Tr((BAB)^k)$ &  $\Tr((A^\prime B^\prime)^k)$ \\
\hline
1&  23.70467 &24\\
2 & 95.90039&96.85024 \\
3&  383.8305&393.54777 \\
\hline
\end{tabular}
\caption{First three moments of $BAB$ and $A^\prime B^\prime$.}
\end{center}
\end{table}
\end{ejem}

\begin{ejem}
Now we give a numerical example for Proposition \ref{P7}. Let $n=300$. Consider the matrices $A = \operatorname{diag}(2^0,2^{-1},2^{-2},\ldots,2^{-n+1})$, $B$, $C$ independent GUE random matrices, and $U$ be an $n\times n$ Haar unitary random matrix, independent of $B$ and $C$. According to Proposition \ref{P7}, the limiting eigenvalues of 
\begin{equation}
\label{matriz}
UBU^*AUCU^* + UCU^*AUBU^*
\end{equation}
are $\lambda_1 \EV(a)\sqcup \lambda_2 \EV(a)$, where $a$ is the operator $\operatorname{diag}(2^0,2^{-1},\ldots,)$ and $\lambda_1,\lambda_2$ are the eigenvalues of the matrix
$$\lim_{n\to\infty}\operatorname{id}\otimes\; \tr_n\begin{pmatrix}
UCU^* UBU^* & UCU^*UCU^*\\UBU^*UBU^* &UBU^*UCU^* 
\end{pmatrix} = \lim_{n\to\infty}  \begin{pmatrix}
\tr_n(CB) & \tr_n(C^2) \\ \tr_n(B^2) & \tr_n(BC) 
\end{pmatrix} = \begin{pmatrix}
1 & 2\\ 2&1
\end{pmatrix},
$$and so $\lambda_1 = 3$ and $\lambda_2 = -1$. Hence, the limiting eigenvalues multiset is
$$\{3\cdot 2^{-n}, -2^{-n}\;:\;n\geq0\}.$$ A numerical realization of the first 15 eigenvalues of the matrix \eqref{matriz} is done in the next picture, where we compare with the theoretical limiting eigenvalues.

\begin{center}
\begin{figure}[H]
\centering\includegraphics[scale=0.4]{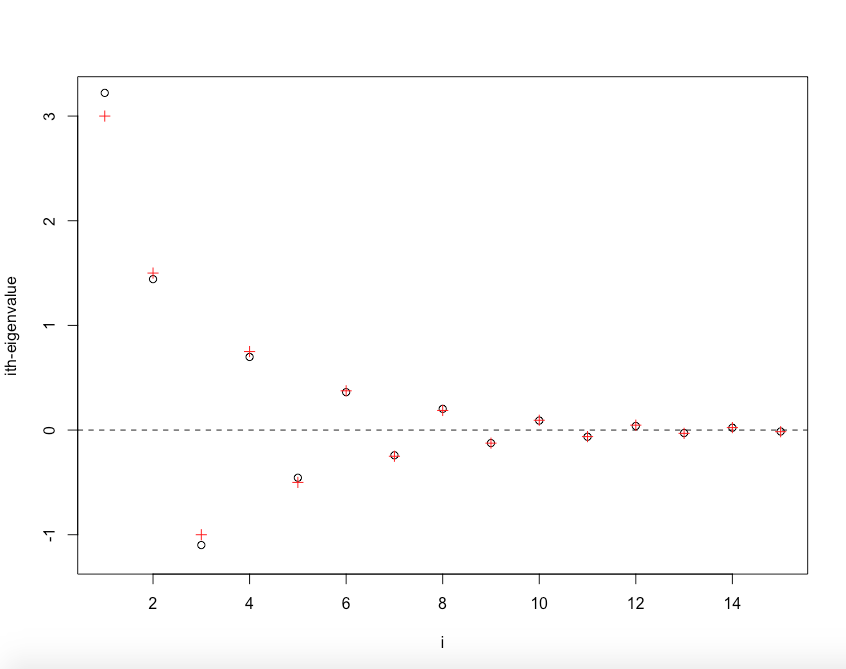}
\caption{Black circles correspond to the eigenvalues of a realization of the matrix \eqref{matriz}. Red crosses correspond to the limiting theoretical eigenvalues.}
\end{figure}
\end{center}
\end{ejem}

\begin{ejem} \label{a+ababa matrix} This is a continuation of Example \ref{a+ababa}. Let $n=300$, $D = \operatorname{diag}(2^{-1},2^{-2},\ldots, 2^{-n})$, $U$ be a Haar unitary random matrix, $A = UDU^*$, $G$ be a GUE random matrix, independent of $U$ and and $B = G^2$. Consider the matrix $X = A+ BABAB$. From Example \ref{a+ababa}, we have that
$$\lim_{n\to\infty} \EV(X) = \EV\left(\sqrt{B^\prime}\operatorname{diag}(a,a^2)\sqrt{B^\prime} \right),$$
where $a$ and $b$ are the limiting operators of $A$ and $B$, respectively, and $B^\prime = \begin{pmatrix} 1&1\\ 1& 2 \end{pmatrix}$.

\begin{center}
\begin{figure}[H]
\centering
\includegraphics[scale=0.45]{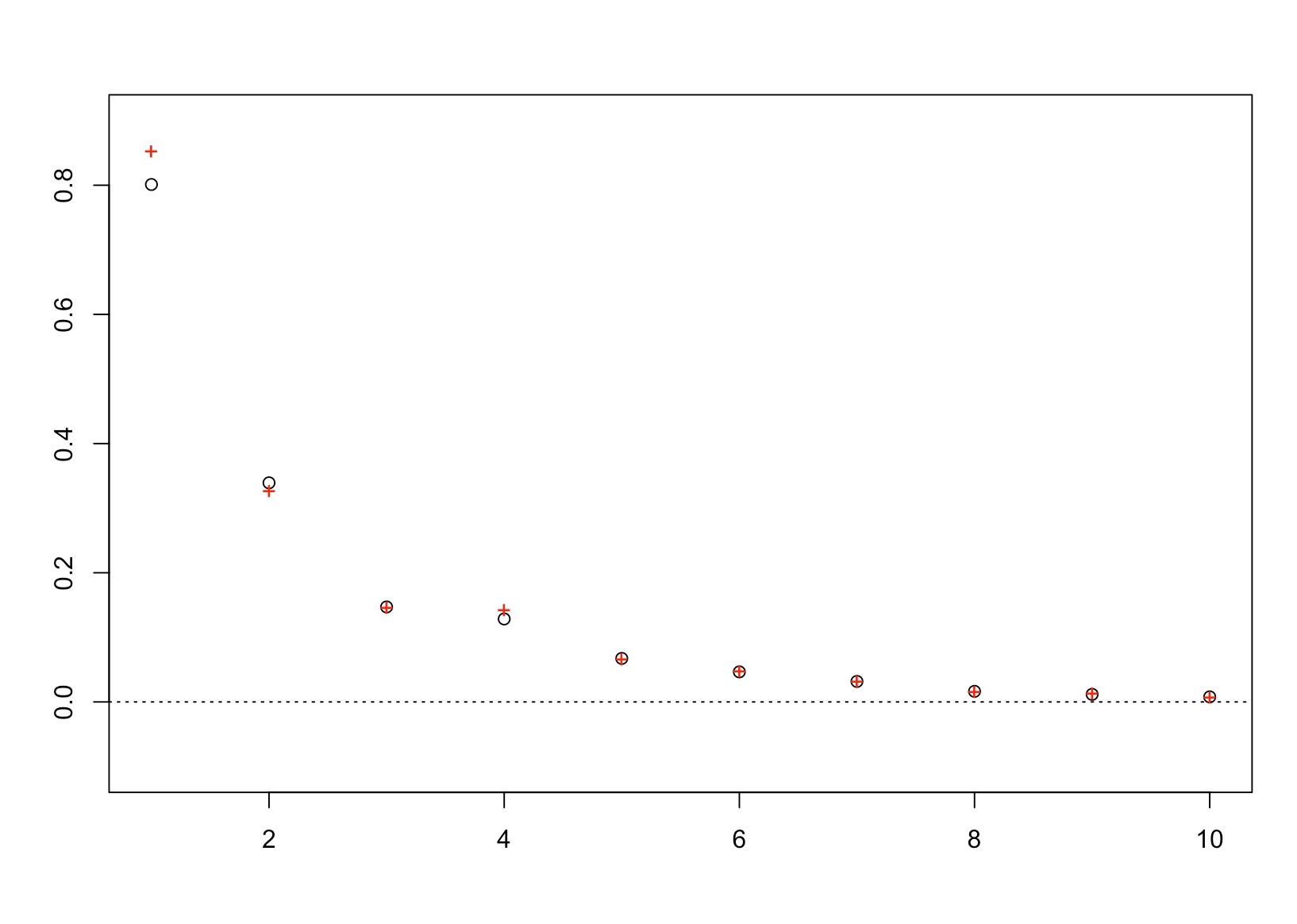}
\caption{Black circles correspond to the eigenvalues of a realization of the matrix $X$. Red crosses correspond to the limiting theoretical eigenvalues.}
\end{figure}
\end{center}
\end{ejem}


\end{document}